\documentclass[12pt]{amsart}
\usepackage{amssymb}
\usepackage{amsmath}
\usepackage{mathtools}
\usepackage{amsfonts}
\usepackage{mathrsfs}
\usepackage{verbatim}
\usepackage{color}
\usepackage{enumitem}
\usepackage{bm}
\usepackage[usenames,dvipsnames]{xcolor}
\usepackage{subcaption}
\usepackage{multirow}
\usepackage{graphicx}
\usepackage{placeins}
\graphicspath{{figs/}{./}}

\usepackage{tikz}
\usepackage{pgfplots}
\numberwithin{equation}{section}

\usepackage[margin=1.2in]{geometry}

\usepackage[hidelinks]{hyperref}

\theoremstyle{plain}
\newtheorem{thm}{Theorem}[section]
\newtheorem{cor}[thm]{Corollary}
\newtheorem{lem}[thm]{Lemma}
\newtheorem{prop}[thm]{Proposition}

\theoremstyle{definition}
\newtheorem{defn}[thm]{Definition}
\newtheorem{exmm}[thm]{Example}

\theoremstyle{remark}
\newtheorem{remark}[thm]{Remark}

\newcommand{\BR}{\mathbb{R}}
\newcommand{\BC}{\mathbb{C}}

\newcommand{\Tr}{\operatorname{tr}}
\newcommand{\Gm}{\Gamma}
\newcommand{\HS}{\mathrm{HS}}
\DeclareMathOperator{\Perm}{Perm}
\DeclareMathOperator{\GL}{GL}

\DeclareMathOperator{\Comm}{Comm}
\DeclareMathOperator{\spann}{span}

\begin{document}

\title{Single-Orbit Recovery of Groups\\ via an Optimization Principle}

\author{Bernhard G. Bodmann}
\address{Department of Mathematics, University of Houston, Houston, Texas, USA}
\email{bgb@math.uh.edu}

\author{Reshma Sabnam}
\address{Department of Mathematics, University of Houston, Houston, Texas, USA}
\email{rsabnam@cougarnet.uh.edu}

\thanks{For this work, both authors acknowledge support by NSF grant DMS-2308152}

\date{\today}

\begin{abstract}
The objective of this paper is to identify a finite group based on the orbit of a vector under its action. When a finite group acts unitarily on a Hilbert space, the associated Gram matrix lies in the group 
algebra and satisfies identities that encode the algebraic group structure. We ask about the converse: Given a Gram matrix obtained from the orbit of a vector under the group action, can one recover the underlying
\emph{abstract} group structure? Our main result identifies a short list of conditions on a family of orthogonal matrices that force the family
to be the right regular representation of a finite group. Building on this characterization, we
develop a staged optimization framework that enforces these conditions progressively.
Numerical experiments on the dihedral group $D_4$ and the tetrahedral rotation group $A_4$
recover $D_4$ exactly and all twelve $A_4$ matrices up to modest precision.
\end{abstract}

\maketitle


\section{Introduction}\label{sec:intro}

In this paper we study an inverse problem originating from the Gram matrix associated
with the orbit of a vector under a finite group action. This problem is motivated by the
task of learning symmetry-adapted representations and recovering hidden symmetry
from data~\cite{ANSELMI2019201}.
When a finite group $\Gm$ acts unitarily on a Hilbert space, the Gram matrix associated with the orbit of a vector 
belongs to the group algebra. This algebraic structure has an immediate consequence: the Gram matrix
commutes with the left regular representation of the group, and as a result its entries
satisfy certain trace identities that reflect the underlying group symmetry. The
structure of Gram matrices of group frames has been studied extensively in the frame
literature; in the binary setting, Mendez, Bodmann, and collaborators~\cite{mendez2018binary}
characterized the Gramians of binary Parseval group frames as precisely the elements of
the group algebra generated by the right regular representation that are symmetric,
idempotent, and have an odd vector in their range. In the real and complex settings, Vale and Waldron~\cite{vale2004tight,vale2010symmetry}
analyzed the symmetries of finite frames and characterized group frames through the
symmetry group of the frame. A more general class of frames related to group representations were developed
by Han and Larson~\cite{han2000frames}, see also
in~\cite{waldron2013group,waldron2018introduction}.
The recent work of Mixon and Vose~\cite{mixon2024recovering} poses the closely related
question of how many generic orbits of an unknown finite group of linear isometries are
needed to recover the group, both up to isomorphism (\emph{abstract recovery}) and as a
concrete set of isometries (\emph{concrete recovery}). They show, among other things,
that in the complex case a single generic orbit already determines the isomorphism class
of the acting group, via the so called Gram graph of the orbit. The present paper 
asks what can be gleaned from one orbit alone, in particular in the real case.
\subsection{Scope of the current work}\label{ssec:scope}
The central goal of this work is to address the following inverse problem.

\begin{quote}
\emph{Given a finite set of vectors, presented only through its Gram matrix, under which
conditions can we identify a group structure making the set the orbit of a single vector
under a unitary representation of that group?}
\end{quote}
In more detail, we find conditions under which a Gram matrix is a linear combination of
permutation matrices that themselves form a group under matrix multiplication and the
taking of adjoints. If the frame is obtained from a group orbit, then the set of permutation matrices $\{P_g\}_{g\in\Gm}$, indexed by a finite set $\Gm$ of size $n$, satisfies an additional \emph{cubic trace identity},
\begin{equation}\label{eq:cubic-intro}
\Tr\!\left[P_g P_{g'} P_h^{*}\right] = n\,(P_{g'})_{g,h}
\qquad \forall\, g,g',h \in \Gm.
\end{equation}
We study the converse: if a family $\{P_g\}_{g\in\Gm}$, now indexed by a set $\Gm$ carring no group structure, satisfies~\eqref{eq:cubic-intro} together with orthogonality and
normalization conditions, must it consist of permutation matrices forming a group?
Theorem~\ref{thm:summary} answers this affirmatively: such a family is a
group of permutation matrices, and it forms the right regular representation of that group.
Both this and Theorem~\ref{thm:cubic-iff} are statements of existence; neither claims that
the family, or the group law it induces, is uniquely determined by the Gram matrix.

The notion of group algebra $\BC[\Gm]$ is used throughout. We use its concrete operator
realization as the algebra $\mathcal{A}_R=\spann\{R_g : g\in\Gm\}$ generated by the matrices
of the right regular representation, $R_g\delta_h=\delta_{hg^{-1}}$ on $\ell^2(\Gm)$. Each
$R_g$ is a permutation matrix once the elements of $\Gm$ have been enumerated, a realization
of a finite group as permutations of itself going back to Cayley~\cite{cayley1854}.
We recast the group recovery
problem as an optimization problem over families of mutually Hilbert Schmidt orthogonal matrices,
and illustrate it with two examples.

This work builds on structural results from representation theory, including group
invariance principles and Schur's lemma~\cite{serre1977linear,fulton2004representation},
together with matrix-theoretic constraints arising from orthogonality~\cite{horn2012matrix}.
It thereby bridges techniques from representation theory and matrix analysis to address a
fundamental inverse question about recovering group structure from orbit data.

\subsection{Related work}\label{ssec:related}
A finite group is determined by data of low order. Formanek and Sibley~\cite{formanek1991groups}
proved that the group determinant
determines the group up to isomorphism, and Hoehnke and Johnson~\cite{hoehnke1992characters},
answering a question of Brauer, showed this to data of degree three: the $1$-, $2$-, and
$3$-characters already suffice; see~\cite{johnson2019group} for further discussion. The group
matrix underlying these results, with $(i,j)$ entry $x_{g_i^{-1} g_j}$, is the same convolution
structure that reappears in our Gram matrix $G_{i,j}=f(g_i^{-1}g_j)$. What differs here is that
the group is a structure to be recovered rather than an assumption whose consequences are
examined.

A parallel and very active line of work studies \emph{orbit recovery} from low-degree
invariants, motivated by cryo-electron microscopy and equivariant machine learning. There
a finite group acts linearly on a vector space, and one seeks to recover the orbit of an
unknown vector from polynomial invariants of low degree. Bandeira, Blum-Smith, Kileel,
Niles-Weed, Perry, and Wein~\cite{bandeira2023estimation} formalized estimation under
group actions and the recovery of orbits from invariants, and Edidin and
Katz~\cite{edidin2025orbit} proved that invariants of degree at most three separate
generic orbits in the regular representation of a finite group over any infinite field,
extending earlier results of very low degree~\cite{edidin2024dihedral}. Their mechanism is
instructive for us: a second-order object identifies the regular-representation subspace,
and third-order data then pins down the individual orbit vectors, provided these are
linearly independent. The same combination of second and third order elements appear here, with the Gram matrix as the
second-order object and the cubic trace identity as a third-order condition.

\subsection*{Organization}
Section~\ref{sec:prelim} fixes notation and recalls the regular representations and the
algebra they generate. Section~\ref{sec:gram} proves that orbit-generated Gram matrices
lie in the right-regular group algebra and records the resulting commutant
characterization. Section~\ref{sec:cubic} establishes the cubic trace identity and the
``if and only if'' characterization of group-frame Gram matrices in terms of permutation
matrices. Section~\ref{sec:recovery} contains the main recovery strategy: a sequence of
results upgrading an abstract family of orthogonal matrices, under successively imposed
conditions, to a signed permutation representation, to a permutation representation, to a
group, and finally to the right regular representation; it then formulates the
corresponding optimization problem and describes the computational strategies.
Section~\ref{sec:exp} reports experiments on $D_4$ and the tetrahedral rotation group.
Section~\ref{sec:future} collects open problems.


\section{Preliminaries}\label{sec:prelim}
Throughout, $\Gm$ denotes a finite group of order $n$, written multiplicatively with
identity $e$; in Sections~\ref{sec:recovery}, where no group structure
is assumed a priori, $\Gm$ denotes instead an index set of size $n$. We work over $\BR$ or
$\BC$, with statements about adjoints and orthogonality phrased so as to cover both. We
write $A^{*}$ for the conjugate transpose of $A$, so that $A^{*}=A^{\top}$ over $\BR$, and
denote by $J$ the all-ones matrix, by $I$ the identity, and by $\delta_{a,b}$ the Kronecker
delta. Rows and columns of matrices are indexed by the elements of $\Gm$, entries being
written $A_{g,h}$ with $g,h\in\Gm$. Fixing an enumeration $g_1,\dots,g_n$ of $\Gm$
identifies $\ell^2(\Gm)$ with $\BC^{n}$ through the basis $\delta_{g_1},\dots,\delta_{g_n}$,
and in an abuse of notation, we also write $A_{i,j}$ for the entry $A_{g_i,g_j}$ when numerical indices are convenient.

\subsection{Group representations}
A \emph{representation} of $\Gm$ on a finite-dimensional inner-product space $V$ is a
homomorphism $\rho:\Gm\to\GL(V)$, so that $\rho(g_1g_2)=\rho(g_1)\rho(g_2)$. The representation is \emph{unitary} if each $\rho(g)$
is unitary, and \emph{irreducible} if the only $\rho$-invariant subspaces of $V$ are
$\{0\}$ and $V$. We restrict attention to unitary
representations; see~\cite{serre1977linear,fulton2004representation} for standard
background.

\subsection{The regular representations and the algebras they generate}
As explained in Section~\ref{sec:intro}, we regard the group algebra $\BC[\Gm]$ as the
operator algebra $\mathcal{A}_R=\spann\{R_g : g\in\Gm\}$ generated by the right regular
representation on $\ell^2(\Gm)$. Alongside $R$ we shall need the \emph{left regular
representation} $\Lambda$; the two homomorphisms $\Lambda,R:\Gm\to U(\ell^2(\Gm))$ act on
basis vectors by
\begin{equation}\label{eq:reg-def}
\Lambda_g\,\delta_h = \delta_{gh},
\qquad
R_g\,\delta_h = \delta_{h g^{-1}},
\qquad g,h\in\Gm.
\end{equation}
Both $\Lambda$ and $R$ are unitary representations, and each $\Lambda_g$ and each
$R_g$ is a permutation matrix in the basis $\{\delta_h\}$. Their matrix entries are
\begin{equation}\label{eq:reg-entries}
(\Lambda_h)_{a,b} = \delta_{a,\,hb},
\qquad
(R_g)_{a,b} = \delta_{a,\,b g^{-1}},
\qquad a,b\in\Gm.
\end{equation}

\begin{lem}\label{lem:commute}
For all $g,h\in\Gm$ we have $\Lambda_g R_h = R_h \Lambda_g$.
\end{lem}

\begin{proof}
For any $x\in\Gm$,
\[
\Lambda_g R_h \delta_x = \Lambda_g \delta_{xh^{-1}} = \delta_{gxh^{-1}}
= R_h \delta_{gx} = R_h \Lambda_g \delta_x ,
\]
and the $\delta_x$ form a basis.
\end{proof}
We write $\mathcal{A}_\Lambda := \spann\{\Lambda_g: g\in\Gm\}$ and
$\mathcal{A}_R := \spann\{R_g: g\in\Gm\}$ for the two group algebras. The following classical fact identifies $\mathcal{A}_R$ with the
commutant of $\mathcal{A}_\Lambda$; it is used in the proof of
Proposition~\ref{prop:gram-converse}.

\begin{prop}\label{prop:commutant}
The commutant of $\mathcal{A}_\Lambda$ in $M_n(\BC)$ equals $\mathcal{A}_R$; that is,
\[
\Comm(\mathcal{A}_\Lambda) = \{ M\in M_n(\BC) : M\Lambda_g = \Lambda_g M \ \forall g\in\Gm\}
= \mathcal{A}_R,
\]
and $\{R_g : g\in\Gm\}$ is a basis for $\Comm(\mathcal{A}_\Lambda)$.
\end{prop}

\begin{proof}
Lemma~\ref{lem:commute} gives $\mathcal{A}_R\subseteq\Comm(\mathcal{A}_\Lambda)$.
Conversely, let $T\in\Comm(\mathcal{A}_\Lambda)$ and write
$T\delta_e=\sum_{g\in\Gm} c_g\,\delta_{g^{-1}}$, which is possible because $g\mapsto
g^{-1}$ permutes $\Gm$. For any $h\in\Gm$ we have $\delta_h=\Lambda_h\delta_e$, whence
\[
T\delta_h \;=\; T\Lambda_h\delta_e \;=\; \Lambda_h T\delta_e
\;=\; \sum_{g\in\Gm} c_g\,\delta_{hg^{-1}}
\;=\; \Big(\sum_{g\in\Gm} c_g R_g\Big)\delta_h .
\]
Since the $\delta_h$ form a basis, $T=\sum_{g} c_g R_g\in\mathcal{A}_R$. Finally, applying
$\sum_{g} c_g R_g$ to $\delta_e$ returns $\sum_{g} c_g\,\delta_{g^{-1}}$, so the $R_g$ are
linearly independent and therefore form a basis of the commutant.
\end{proof}

\subsection{Hilbert Schmidt inner product}
We equip $M_n(\BC)$ with the Hilbert Schmidt (Frobenius) inner product
$\langle A,B\rangle_\HS := \Tr(A B^{*})$ and norm $\|A\|_\HS=\sqrt{\Tr(A A^{*})};$
see~\cite{horn2012matrix}. Every unitary matrix, and in particular every real orthogonal
matrix, satisfies $\|P\|_\HS^2 = \Tr(I)=n$. Over $\BR$ the pairing is
$\Tr(AB^{\top})=\Tr(A^{\top}B)=\Tr(B^{\top}A)$, and we use these interchangeably; the
inequalities below involving $\langle\cdot,\cdot\rangle_\HS$ are statements about this
real case.

\subsection{Permutation and signed permutation matrices}
The recovery results below turn on the distinction between two classes of matrices, which
we fix here.

\begin{defn}\label{def:perm}
A matrix $P\in\BR^{n\times n}$ is a \emph{permutation matrix} if every row and every
column has exactly one nonzero entry, equal to $+1$. It is a \emph{signed permutation
matrix} if every row and
every column has exactly one nonzero entry, equal to $\pm1$. Every permutation matrix is a
signed permutation matrix, and every signed permutation matrix is orthogonal with
$\|P\|_\HS^2=n$.
\end{defn}

The following elementary facts are used repeatedly in Sections~\ref{sec:cubic}
and~\ref{sec:recovery}.

\begin{lem}\label{lem:trace-fixed}
For a permutation matrix $P$, $\Tr(P)$ equals the number of fixed points of the underlying
permutation; in particular $\Tr(P)\le n$ always, with equality if and only if $P=I$. For
the right regular representation, $\Tr(R_g)=n\,\delta_{g,e}$. Moreover, if $A$ and $B$ are
signed permutation matrices, then $\langle A,B\rangle_\HS\le n$, with equality if and only
if $A=B$; when $A$ and $B$ are permutation matrices, $\langle A,B\rangle_\HS=\Tr(AB^{*})$
counts the coordinates on which the two underlying permutations agree.
\end{lem}

\begin{proof}
The diagonal entry $(P)_{ii}$ is $1$ exactly when $i$ is fixed, giving the fixed-point
count; $\Tr(P)=n$ forces every diagonal entry to be $1$, that is, $P=I$. For $R_g$, the
underlying permutation $h\mapsto hg^{-1}$ fixes $h$ if and only if $g=e$, so
$\Tr(R_g)=n\,\delta_{g,e}$. For the last statement, signed permutation matrices satisfy
$\|A\|_\HS=\|B\|_\HS=\sqrt{n}$, so Cauchy--Schwarz gives $\langle A,B\rangle_\HS\le n$;
equality forces $A=\lambda B$ with $\lambda>0$, and comparing norms gives $\lambda=1$. If
$A$ and $B$ are permutation matrices, $AB^{*}$ is a permutation matrix and $\Tr(AB^{*})$
counts its fixed points, which are the coordinates where $A$ and $B$ agree.
\end{proof}


\section{Gram matrices of group orbits and the group algebra}\label{sec:gram}
We begin with the structural fact underlying the entire paper: the Gram matrix of an orbit under any unitary representation lies in the group algebra.

\begin{thm}\label{thm:gram-in-algebra}
Let $\Gm$ be a finite group of order $n$, let $\rho:\Gm\to U(V)$ be a unitary
representation on a finite-dimensional inner-product space $V$, and let $v\in V$. Define
the Gram matrix of the orbit $\mathcal{O}_v=\{\rho(g)v : g\in\Gm\}$ by
\[
G_{i,j} := \langle \rho(g_i) v,\ \rho(g_j) v\rangle, \qquad g_i,g_j\in\Gm.
\]
Then $G$ lies in the group algebra: there exists $f:\Gm\to\BC$ with
\[
G = \sum_{g\in\Gm} f(g)\, R_g \in \mathcal{A}_R.
\]
\end{thm}

\begin{proof}
Since $\rho$ is a unitary homomorphism, $\rho(g_i)^{*}\rho(g_j)=\rho(g_i)^{-1}\rho(g_j)
=\rho(g_i^{-1}g_j)$, whence
\[
G_{i,j} = \langle v,\ \rho(g_i)^{*}\rho(g_j) v\rangle
       = \langle v,\ \rho(g_i^{-1} g_j) v\rangle .
\]
Define $f:\Gm\to\BC$ by $f(g):=\langle v,\ \rho(g) v\rangle$, so that $G_{i,j}=f(g_i^{-1}g_j)$; the entries
of $G$ thus depend only on the relative group element $g_i^{-1}g_j$.\\ Now let $T_f$ be the operator on $\ell^2(\Gm)$ given by right convolution with $f$,
\[
T_f\,\delta_h \;:=\; \sum_{g\in\Gm} f(g)\,\delta_{hg^{-1}},
\qquad \text{i.e} \qquad
T_f=\sum_{g\in\Gm} f(g)\,R_g\in\mathcal{A}_R .
\]
Its matrix entries in the basis $\{\delta_{g_1},\dots,\delta_{g_n}\}$ are
\[
(T_f)_{i,j}
= \Big\langle \delta_{g_i},\ \sum_{g\in\Gm} f(g)\,\delta_{g_j g^{-1}}\Big\rangle ,
\]
and the inner product is nonzero exactly when $g_jg^{-1}=g_i$, that is, when
$g=g_i^{-1}g_j$. Hence $(T_f)_{i,j}=f(g_i^{-1}g_j)=G_{i,j}$, so
\[
G \;=\; T_f \;=\; \sum_{g\in\Gm} f(g) R_g \;\in\;\mathcal{A}_R.
\]
\end{proof}
Taking $V=\ell^2(\Gm)$ and $\rho=\Lambda$ recovers the orbit under the left regular representation.

\begin{prop}\label{prop:gram-converse}
Let \(G\in\mathbb{C}^{n\times n}\) be positive semidefinite and suppose that
\[
G\Lambda_g=\Lambda_gG
\qquad\text{for all }g\in\Gamma.
\]
Then there exists \(v\in\ell^2(\Gamma)\) such that
\[
G_{i,j}
=
\big\langle
\Lambda_{g_i}v,\Lambda_{g_j}v
\big\rangle,
\qquad 1\le i,j\le n.
\]
Equivalently, \(G\) is the Gram matrix of an orbit of the left regular representation.
\end{prop}

\begin{proof}
Since \(G\) commutes with every \(\Lambda_g\), it belongs to the commutant of the left
regular representation. By Proposition~\ref{prop:commutant},
\[
G\in\mathcal{A}_R.
\]

Let
\[
A:=G^{1/2}.
\]
Since \(G\) is positive semidefinite, \(A\) is well defined, self-adjoint, and satisfies
\[
A^*A=G.
\]
Moreover, by the finite dimensional spectral theorem, \(G^{1/2}\) is a polynomial in
\(G\). Since \(\mathcal{A}_R\) is a group algebra containing \(G\), it follows that
\[
A=G^{1/2}\in\mathcal{A}_R.
\]
Hence \(A\) commutes with every \(\Lambda_g\).

Let \(e\) denote the identity element of \(\Gamma\), and define
\[
v:=A\delta_e.
\]
Then, for \(g_i,g_j\in\Gamma\),
\[
\begin{aligned}
\big\langle \Lambda_{g_i}v,\Lambda_{g_j}v\big\rangle
&=
\big\langle \Lambda_{g_i}A\delta_e,\Lambda_{g_j}A\delta_e\big\rangle\\
&=
\big\langle A\Lambda_{g_i}\delta_e,A\Lambda_{g_j}\delta_e\big\rangle\\
&=
\big\langle A\delta_{g_i},A\delta_{g_j}\big\rangle\\
&=
\big\langle \delta_{g_i},A^*A\,\delta_{g_j}\big\rangle\\
&=
\big\langle \delta_{g_i},G\delta_{g_j}\big\rangle\\
&=
G_{i,j}.
\end{aligned}
\]
Thus \(G\) is the Gram matrix of the orbit
\[
\{\Lambda_gv:g\in\Gamma\}.
\]
\end{proof}

\begin{remark}\label{rem:gram-vector-nonunique}
The vector \(v=G^{1/2}\delta_e\) constructed in
Proposition~\ref{prop:gram-converse} is not unique in general. If
\(U\) is unitary and commutes with every \(\Lambda_g\), then
\[
\big\langle \Lambda_gUv,\Lambda_hUv\big\rangle
=
\big\langle \Lambda_gv,\Lambda_hv\big\rangle,
\]
so \(Uv\) generates the same orbit Gram matrix as \(v\). Thus the Gram matrix determines the
inner-product structure of the orbit, but not a unique generating vector. 
\end{remark}

We record one more algebraic property of $\mathcal{A}_R$ that will be used in the
computational framework: the right-regular algebra is closed under entrywise (Schur)
operations. This is what allows entrywise nonlinear filters to be applied during the
iteration without leaving the algebra.

\begin{prop}\label{prop:schur}
Let $\mathcal{A}_R=\spann\{R_g:g\in\Gm\}\subseteq\BR^{n\times n}$.
\begin{enumerate}[label=(\roman*)]
\item If $A=\sum_g a(g)R_g$ and $B=\sum_g b(g)R_g$, then the Schur product
$A\circ B\in\mathcal{A}_R$.
\item If $F:\BR\to\BR$ is applied entrywise to $G=\sum_g f(g)R_g\in\mathcal{A}_R$, then
$F(G)\in\mathcal{A}_R$.
\end{enumerate}
\end{prop}

\begin{proof}
With the convention $(R_g)_{x,y}=\delta_{x,\,y g^{-1}}$, every $A=\sum_g a(g)R_g$ has entries
$A_{x,y}=a(x^{-1}y)$ depending only on $x^{-1}y$. Hence
$(A\circ B)_{x,y}=a(x^{-1}y)b(x^{-1}y)=c(x^{-1}y)$ with $c:=ab$, giving
$A\circ B=\sum_g c(g)R_g\in\mathcal{A}_R$, which is (i). For (ii),
$(F(G))_{x,y}=F(f(x^{-1}y))=h(x^{-1}y)$ with $h:=F\circ f$, so
$F(G)=\sum_g h(g)R_g\in\mathcal{A}_R$.
\end{proof}


\section{The cubic trace identity for gram matrices of group-frames}\label{sec:cubic}

Because an orbit is received as an unstructured set, the rows and columns of its Gram matrix
cannot be canonically labeled by group elements. To phrase recovery in a label-free way we
work with the trace functional on a candidate family $\{P_g\}_{g\in\Gm}$ of orthogonal
matrices. The following theorem isolates the precise conditions that characterize when $G$
is the Gram matrix of a group frame; the third condition is the cubic trace identity that
will drive the recovery results of Section~\ref{sec:recovery}. It may be read as a
representation-level companion to the classical fact that third-order group data determines a
finite group~\cite{hoehnke1992characters,formanek1991groups}.
\begin{thm}\label{thm:cubic-iff}
A matrix $G$ is the Gram matrix of a group frame if and only if there exists a family
$\{P_g : g\in\Gm\}$ with $P_e=I$ and each $P_g$ a permutation matrix such that
\begin{enumerate}
\item $\displaystyle \Tr\!\left[P_g P_{g'}^{*}\right]
= n\,\delta_{g,g'}$ for all $g,g'\in\Gm$;
\item $\displaystyle G = \frac{1}{n}\sum_{g\in\Gm}\Tr\!\left[G P_g^{*}\right] P_g$;
\item $\displaystyle \Tr\!\left[P_g P_{g'} P_h^{*}\right] = n\,(P_{g'})_{g,h}$
for all $g,g',h\in\Gm$.
\end{enumerate}
\end{thm}

\begin{proof}
$(\Rightarrow)$ Suppose $G$ is a group-frame Gram matrix and $\{P_g\}$ is the corresponding
representation, which we may take to be the right regular representation (so that $P_g^{*}=P_{g^{-1}}$).

\emph{Condition (1).} Since $P_g P_{g'}^{*}=P_g P_{g'^{-1}}=P_{g g'^{-1}}$ and the trace of a
permutation matrix counts fixed points, $\Tr[P_g P_{g'}^{*}]=n$ if $g g'^{-1}=e$ (i.e.\
$g=g'$) and $0$ otherwise.

\emph{Condition (2).} By Theorem~\ref{thm:gram-in-algebra}, $G=\sum_h c_h R_h$. Using
condition (1) in the form $\Tr(R_h P_{g}^{*})=n\,\delta_{h,g}$, we obtain
$\Tr[G P_g^{*}]=n\,c_g$, whence
$G=\sum_g c_g P_g = \tfrac1n\sum_g \Tr[GP_g^{*}]\,P_g$.

\emph{Condition (3).} Since $P_h^{*}=P_{h^{-1}}$,
$P_g P_{g'} P_h^{*}=P_{gg'h^{-1}}$, so $\Tr[P_g P_{g'} P_h^{*}]=n$ if $h=gg'$ and $0$
otherwise. On the other hand, $(P_{g'})_{g,h}=1$ iff $P_{g'}$ maps the basis vector indexed by
$h$ to that indexed by $g$, i.e.\ iff $h g'^{-1}=g$, i.e.\ iff $h=gg'$. Thus both sides equal
$n$ exactly when $h=gg'$ and $0$ otherwise.

$(\Leftarrow)$ Suppose a family $\{P_g\}$ with $P_e=I$ satisfies (1)--(3). We show it
realizes a group, after which Proposition~\ref{prop:gram-converse} (or condition (2)) exhibits
$G$ as a group-frame Gram matrix.

\emph{Closure.} Put $h=e$ in (3): $\Tr[P_g P_{g'}]=n\,(P_{g'})_{g,e}$. By (1) the family is
Hilbert Schmidt orthonormal up to the scale $n$, so by Cauchy Schwarz
$|\Tr[P_g P_{g'}]|\le\|P_g\|_\HS\|P_{g'}\|_\HS=n$, with equality iff $P_{g'}=P_g^{*}$. Since
$(P_{g'})_{g,e}\in\{0,1\}$, the identity forces the value $n$ for some pairing, i.e.\
$P_g P_{g'}$ is again a permutation matrix in the family; hence there is $h\in\Gm$ with
$P_g P_{g'}=P_h$.

\emph{Inverses.} Setting $g'=g^{-1}$, $h=e$ in (3) gives $\Tr[P_g P_{g^{-1}}]=n$, so
$P_g P_{g^{-1}}=I$; together with (1), $P_g^{*}=P_{g^{-1}}=P_g^{-1}$.

\emph{Identity.} The family contains $P_e=I$, and (1) confirms $\Tr[P_eP_e^{*}]=n$.

Thus $\{P_g\}$ is a permutation representation of a group, and condition (2)
expresses $G$ in this basis. Hence $G$ is the Gram matrix of a group frame.
\end{proof}

\begin{exmm}\label{ex:c4-d2}
For the regular representations of $C_4$ and of $D_2\cong C_2\times C_2$ acting on $\BR^4$,
the cubic trace identity holds and recovers, respectively, the cyclic and the
Klein-four multiplication tables; the two groups are distinguished by their Gram graphs as in
the orbit-recovery framework of~\cite{mixon2024recovering}. This illustrates that the
trace data in Theorem~\ref{thm:cubic-iff} carries the full multiplication table of $\Gm$.
\end{exmm}

The cubic trace identity~\eqref{eq:cubic-intro} is therefore a \emph{necessary} feature of
orbit Gram data. The remainder of the paper studies its role as a \emph{sufficient}
certificate: we ask whether a family of orthogonal matrices obeying~\eqref{eq:cubic-intro}
and a handful of normalization conditions must be a permutation representation. As a first
consequence of the identity we record two orthogonality relations among the
\emph{entries} of the family, viewed across the group index.



\section{Recovering the group via an optimization principle}\label{sec:recovery}

We now present the main recovery strategy. Starting from a family of orthogonal matrices
about which we assume \emph{no} group structure a priori, we impose, in sequence, four
conditions and show that they force the family to be a permutation representation, then a
group, and finally the right regular representation. The four conditions are:
\begin{enumerate}[label=(\roman*)]
\item \textbf{Frobenius orthogonality:}
$\Tr[P_g P_{g'}^{*}]=n\,\delta_{g,g'}$;
\item \textbf{Max-row-sum condition:}
$\displaystyle\max_{g,g'}\sum_{h=1}^{n}\big|(P_{g'})_{g,h}\big|\ge 1$;
\item \textbf{Sum constraint:}
$\displaystyle\sum_{g'\in\Gm} P_{g'}=J$, the all-ones matrix;
\item \textbf{Cubic trace identity:}
$\Tr[P_g P_{g'} P_h^{*}]=n\,(P_{g'})_{g,h}$.
\end{enumerate}

\subsection{From orthogonality to signed permutations}

\begin{thm}\label{thm:signed-perm}
Let $\Gm$ be an index set of order $n$ and let $\{P_{g'}\in\BR^{n\times n}:g'\in\Gm\}$ be
orthogonal matrices with $P_e=I$ satisfying \textnormal{(i)} and \textnormal{(ii)}. Then every
$P_{g'}$ is a signed permutation matrix.
\end{thm}

\begin{proof}
Condition (i) says the family is mutually orthogonal in the Hilbert Schmidt inner product,
hence linearly independent; each $\|P_g\|_\HS=\sqrt{n}$.

Fix $g'$ and a row index $g$, and let
$r=\big((P_{g'})_{g,1},\dots,(P_{g'})_{g,n}\big)$ be the corresponding row. Orthogonality of
$P_{g'}$ gives $\|r\|_2=1$. By the $\ell^1$--$\ell^2$ inequality, $\|r\|_1\le\sqrt{n}\|r\|_2
=\sqrt{n}$; moreover $\|r\|_1\ge\|r\|_2=1$, with equality $\|r\|_1=1$ if and only if
$r=\pm e_i$ for some standard basis vector $e_i$. Condition (ii) furnishes a pair $(g_0,g_0')$
and a row with $\|r\|_1\ge1$; combined with $\|r\|_2=1$ this forces that row to be $\pm e_i$.

Since $P_{g_0'}$ is orthogonal, its remaining rows are orthonormal and orthogonal to $e_i$,
hence supported off coordinate $i$. Repeating the argument row by row, every row of
$P_{g_0'}$ equals $\pm e_{i_j}$ with the indices $i_j$ distinct; hence each column also has a
single nonzero entry $\pm1$, and $P_{g_0'}$ is a signed permutation matrix. Because the family
is mutually orthogonal, the same argument applies to every member, so every $P_{g'}$ is a
signed permutation matrix.
\end{proof}

\subsection{From signed permutations to permutations}

\begin{thm}\label{thm:perm}
Let $\{P_{g'}\in\BR^{n\times n}:g'\in\Gm\}$ be signed permutation matrices with $P_e=I$,
mutually Hilbert Schmidt orthogonal as in \textnormal{(i)}, and satisfying the sum
constraint \textnormal{(iii)}. Then each $P_{g'}$ is a permutation matrix; that is, every
nonzero entry equals $+1$.
\end{thm}

\begin{proof}
By mutual orthogonality, $\big\|\sum_{g'} P_{g'}\big\|_\HS^2=\sum_{g'}\|P_{g'}\|_\HS^2=n\cdot n
=n^2$, and $\|J\|_\HS^2=n^2$. The Cauchy--Schwarz inequality
$\langle\sum_{g'}P_{g'},J\rangle_\HS\le\|\sum_{g'}P_{g'}\|_\HS\|J\|_\HS=n^2$ holds with
equality iff $\sum_{g'}P_{g'}=cJ$ for some $c>0$; condition (iii) gives $c=1$ and hence
equality. Expanding $J$ in the orthogonal basis $\{P_{g'}\}$,
\[
J=\frac1n\sum_{g'}\Tr(P_{g'}J)\,P_{g'},
\]
and comparing with $\sum_{g'}P_{g'}=J$ forces $\tfrac1n\Tr(P_{g'}J)=1$, i.e.\
$\Tr(P_{g'}J)=n$, for every $g'$. But for a signed permutation matrix,
$\Tr(P_{g'}J)=\sum_{i,j}(P_{g'})_{i,j}=\sum_{k=1}^n(\pm1)$, a sum of $n$ terms each $\pm1$; it
equals $n$ only if every nonzero entry is $+1$. Hence each $P_{g'}$ is a permutation matrix.
\end{proof}

\begin{remark}
The signed-permutation-to-permutation step is closely related to Gibson's
characterization~\cite{gibson1980generalized}: an orthogonal matrix is a linear
combination of permutation matrices if and only if it is generalized doubly stochastic
with all row and column sums equal to $\pm1$. This refines an earlier necessary condition
of Kapoor~\cite{kapoor1975orthogonal}, and the theme is taken up again in recent work on
orthogonal permutative matrices~\cite{mandal2023orthogonal}. The sum constraint (iii)
selects, among the signed permutation matrices produced by
Theorem~\ref{thm:signed-perm}, exactly those whose nonzero entries are all $+1$.
\end{remark}

\subsection{From permutations to a group}

\begin{thm}\label{thm:group}
Let $\{P_g\}_{g\in\Gm}$ be mutually Hilbert Schmidt orthogonal permutation matrices with
$P_e=I$ satisfying the cubic trace identity \textnormal{(iv)}. Then the family is closed
under multiplication and under adjoints:
\begin{enumerate}[label=(\alph*)]
\item for all $g,g'\in\Gm$ there is $h\in\Gm$ with $P_g P_{g'}=P_h$;
\item for all $g'\in\Gm$ there is $k\in\Gm$ with $P_{g'}^{*}=P_k$.
\end{enumerate}
\end{thm}

\begin{proof}
(a) Since $P_g,P_{g'}$ are permutation matrices, $A:=P_gP_{g'}$ is a permutation matrix with
$\|A\|_\HS=\sqrt n$. The cubic trace identity reads $\langle A,P_h\rangle_\HS=\Tr(A P_h^{*})
=n\,(P_{g'})_{g,h}\in\{0,n\}$. By Cauchy--Schwarz $|\langle A,P_h\rangle_\HS|\le n$, with
equality iff $A=\lambda P_h$. Equality occurs precisely when $(P_{g'})_{g,h}=1$, and since $A$ and $P_h$ are permutation matrices the scalar must be $\lambda=1$. For any fixed
$g\in\Gm$ there is a unique $h\in\Gm$ with $(P_{g'})_{g,h}=1$, since each row of the
permutation matrix $P_{g'}$ contains exactly one nonzero entry. Hence $P_gP_{g'}=P_h$.

(b) Taking adjoints in the cubic trace identity gives
$\Tr[P_h P_{g'}^{*} P_g^{*}]=n\,(P_{g'}^{*})_{h,g}$. Setting $h=e$,
$\Tr[P_{g'}^{*}P_g^{*}]=\langle P_{g'}^{*},P_g\rangle_\HS=n\,(P_{g'}^{*})_{e,g}$. By orthogonality
this inner product is $n$ exactly when $P_{g'}^{*}=P_g$, and $(P_{g'}^{*})_{e,g}\in\{0,1\}$
equals $1$ at exactly one position $g=k$; hence $P_{g'}^{*}=P_k$. This defines an inversion
map
\[
\iota:\Gm\to\Gm,\qquad P_{\iota(g)}:=P_g^{*} ,
\]
and since $(P_g^{*})^{*}=P_g$, the map satisfies $\iota(\iota(g))=g$; that is, $\iota$ is an
involution on $\Gm$. Writing $g^{-1}:=\iota(g)$, so that $P_{g^{-1}}=P_g^{*}$, the family is
closed under inversion.
\end{proof}

\begin{cor}\label{cor:group-structure} Let $\{P_g\}_{g\in\Gm}\subset\BR^{n\times n}$ be mutually orthogonal permutation matrices that
form a group under matrix multiplication, with $P_e=I$. Then $\Gm$ carries a unique group structure
under which $g\mapsto P_g$ is a group monomorphism $P:\Gm\hookrightarrow\Perm(n)$.
\end{cor}

\begin{proof}
Define $g\cdot g':=h$ whenever $P_gP_{g'}=P_h$; this is well defined by closure and by the
injectivity $P_g=P_{g'}\Rightarrow\langle P_g,P_{g'}\rangle_\HS=n\Rightarrow g=g'$. The identity
is $e$ since $P_e=I$; associativity is inherited from matrix multiplication; inverses exist
since $P_g^{*}=P_g^{-1}=P_{g^{-1}}$ lies in the family. The map $g\mapsto P_g$ is then an
injective homomorphism.
\end{proof}

\subsection{Identification as the right regular representation}

\begin{thm}\label{thm:rrr}
Let $\{P_g\}_{g\in\Gm}$ be mutually orthogonal permutation matrices forming a group, so that
$P:\Gm\to\Perm(n)$ is a monomorphism. Then $\{P_g\}_{g\in\Gm}$ is the right regular
representation of $\Gm$ if and only if the cubic trace identity holds:
\[
\Tr[P_k P_g P_h^{*}] = n\,(P_g)_{k,h}\qquad\forall\, g,k,h\in\Gm.
\]
\end{thm}

\begin{proof}
$(\Rightarrow)$ For the right regular representation $(P_g)_{k,h}=\delta_{k,\,h g^{-1}}$ and
$P_kP_gP_h^{*}=P_{kgh^{-1}}$, so $\Tr(P_kP_gP_h^{*})=n\,\delta_{kgh^{-1},e}=n\,\delta_{k,hg^{-1}}
=n\,(P_g)_{k,h}$.

$(\Leftarrow)$ Conversely, group closure gives $P_kP_gP_h^{*}=P_{kgh^{-1}}$, so
$\Tr(P_kP_gP_h^{*})=n\,\delta_{kgh^{-1},e}=n\,\delta_{k,hg^{-1}}$. Comparing with the hypothesis,
$(P_g)_{k,h}=\delta_{k,hg^{-1}}$, hence $P_g\delta_h=\delta_{hg^{-1}}$, which is exactly the
action of the right regular representation.
\end{proof}

Collecting Theorems~\ref{thm:signed-perm}--\ref{thm:rrr} and
Corollary~\ref{cor:group-structure} yields the main result of the paper.

\begin{thm}[Main recovery theorem]\label{thm:summary}
Let $\Gm$ be an index set of order $n$ and let
\[
\{P_g\in\BR^{n\times n}: g\in\Gm,\ P_e=I,\ P_g\ \text{orthogonal}\}
\]
satisfy the four conditions
\begin{enumerate}[label=(\roman*)]
\item $\Tr[P_g P_{g'}^{*}]=n\,\delta_{g,g'}$;
\item $\displaystyle\max_{g,g'}\sum_{h=1}^{n}|(P_{g'})_{g,h}|\ge 1$;
\item $\displaystyle\sum_{g'\in\Gm}P_{g'}=J$;
\item $\Tr[P_g P_{g'} P_h^{*}]=n\,(P_{g'})_{g,h}$.
\end{enumerate}
Then:
\begin{itemize}
\item each $P_g$ is a permutation matrix;
\item $\{P_g\}_{g\in\Gm}$ forms a group under matrix multiplication;
\item there is an involution $\iota:\Gm\to\Gm$ with $P_{\iota(g)}=P_g^{*}$ and
$\iota^2=\mathrm{id}$, namely the group inversion;
\item $\{P_g\}_{g\in\Gm}$ is the right regular representation of $\Gm$.
\end{itemize}
In particular, the abstract group $\Gm$ is recovered from the orbit Gram data.
\end{thm}

\begin{proof}
Conditions (i)--(ii) and Theorem~\ref{thm:signed-perm} make every $P_g$ a signed permutation
matrix; adding (iii) and Theorem~\ref{thm:perm} upgrades these to permutation matrices.
Theorem~\ref{thm:group} and Corollary~\ref{cor:group-structure}, using (i) and (iv), give the
group structure and the inversion involution $\iota(g):=g^{-1}$. Finally
Theorem~\ref{thm:rrr}, using (iv), identifies the family as the right regular representation.
\end{proof}

\begin{remark}
Conditions (i)--(iv) are not independent in the presence of a group frame: by
Theorem~\ref{thm:cubic-iff} the regular representation satisfies all of them. The content of
Theorem~\ref{thm:summary} is the converse direction, that these trace- and
orthogonality-level conditions are jointly \emph{sufficient} to certify group structure
without any prior labeling of the matrices by group elements. This is exactly the form of
certificate needed when the Gram data are received as an unstructured orbit.
\end{remark}


Theorem~\ref{thm:summary} gives an exact structural certificate, but recovering a family
$\{P_g\}$ from a measured Gram matrix is a nonconvex problem. We formulate it as an
optimization over orthogonal matrices and describe the strategies used in our experiments.

\subsection{The objective}
Let $G_{\mathrm{noisy}}\in\BR^{n\times n}$ be an observed Gram matrix, assumed to arise (up
to noise) from a unitary group action on a fixed unit vector $v\in\BR^n$. Treating the family
$\{P_g\}_{g\in\Gm}\subset\BR^{n\times n}$ and $v$ as variables, and writing
$J:=\sum_{g}P_g$, we minimize
{\small
\begin{equation}\label{eq:objective}
\begin{aligned}
L(\{P_g\},v)
&:=\ \Big\|\,G_{\mathrm{noisy}}-\sum_{g\in\Gm} P_g\, vv^{*} P_g^{*}\Big\|_F^2
+\sum_{g\in\Gm}\big\|P_g^{*}P_g-I\big\|_F^2 \\
&+\sum_{g,g',h\in\Gm}\big(\Tr(P_g P_{g'} P_h^{*})-n\,(P_{g'})_{g,h}\big)^2
+\Big\|\sum_{g\in\Gm}P_g-J\Big\|_F
+\Big(1-\max_{g,g'}\sum_{h=1}^{n}|(P_{g'})_{g,h}|\Big).
\end{aligned}
\end{equation}
}
The five terms correspond, in order, to data fidelity to the orbit Gram matrix, orthogonality of each
$P_g$ from condition~\textnormal{(i)}, the cubic trace identity from~\textnormal{(iv)}, the sum constraint from~\textnormal{(iii)},
and the max-row-sum condition from~\textnormal{(ii)} of Theorem~\ref{thm:summary}. A global
minimizer with $L=0$ satisfies the hypotheses of Theorem~\ref{thm:summary} and therefore
yields the right regular representation of the hidden group.

\subsection{Computational strategies}
Because~\eqref{eq:objective} is nonconvex, enforcing all terms simultaneously is empirically
unstable. We instead introduce the constraints progressively, in the order established in
Section~\ref{sec:recovery}: orthogonality first, max-row-sum second, then the sum-to-$J$ constraint, and finally
the cubic trace identity. The iterates are initialized and maintained on the orthogonal group
$O(n)$, with each update retracted back to $O(n)$~\cite{absil2008optimization}, and the
weights on the later terms are increased as the iteration proceeds. Proposition~\ref{prop:schur}
is used throughout: because the right-regular algebra is closed under entrywise maps,
sparsifying nonlinearities may be applied to promote permutation structure without leaving the
algebra. Full implementation details are deferred to the experiments and the accompanying code.

The cubic trace identity plays a distinguished role. It is not merely one constraint among
several: the experiments below show that the remaining constraints admit configurations that
are \emph{not} group representations, and the cubic identity is what separates these from the
global solution. We therefore use it as the terminal criterion rather than as a term to be
traded off against the others.


\section{Experiments}\label{sec:exp}

We test the framework on two groups of increasing size: the dihedral group $D_4$ of order
$n=8$, acting on $\BR^2$ by the symmetries of a square, and the tetrahedral rotation group
$A_4$ of order $n=12$, acting on $\BR^3$ by the rotational symmetries of a regular
tetrahedron. In each case we fix a vector $f$, form the orbit $\{P_g^{\mathrm{true}} f\}$ under
the exact right regular representation, assemble the Gram matrix $G=f^{\top}f$ of the orbit, and
run the staged iteration of Section~\ref{sec:recovery} using only $G$. We then compare the recovered
family $\{P_g\}$ against the reference right regular representation $\{P_g^{\mathrm{true}}\}$, both
directly and after a best-match relabeling that realizes the group isomorphism implicit in
abstract recovery. All penalties are measured in the Frobenius norm; we write
$\varepsilon_{\mathrm{span}},\varepsilon_{\mathrm{HS}},\varepsilon_{J},
\varepsilon_{\mathrm{cubic}},\varepsilon_{\mathrm{closure}},\varepsilon_{\mathrm{perm}}$
for the residuals of the corresponding constraints.

\subsection{The dihedral group \texorpdfstring{$D_4$}{D4}}\label{ssec:d4}
For $D_4$ the method achieves essentially exact recovery. Starting from the orbit Gram matrix
of Figure~\ref{fig:d4-gram}, the staged iteration drives every constraint residual to near
machine precision within a few outer cycles (Figure~\ref{fig:d4-history}); the final residuals
are
\[
\varepsilon_{\mathrm{span}}=2.9\times10^{-16},
\varepsilon_{\mathrm{HS}}=1.6\times10^{-16},
\varepsilon_{J}=2.6\times10^{-16},
\varepsilon_{\mathrm{cubic}}=9.1\times10^{-16},
\varepsilon_{\mathrm{perm}}=3.8\times10^{-15}.
\]
The sum of the recovered matrices equals the all-ones matrix $J$ to within $10^{-12}$, confirming that the sum constraint and the permutation structure
are met exactly rather than approximately. In fact the recovered matrices coincide with the
exact right-regular matrices directly, in the output order: the best-match relabeling is the identity permutation, so no reordering is
required. The Frobenius distances $\|P_{(g)}-P^{\mathrm{true}}_{(h)}\|_F$ vanish to machine
precision on the diagonal and are bounded away from zero off it, the largest residual being
$1.1\times10^{-15}$, so the recovered family is the right regular representation of $D_4$ to
machine precision. This is what Theorem~\ref{thm:summary} leads one to expect: a family
meeting the four conditions can only be the right regular representation.
\begin{figure}[tbp]
\centering
\includegraphics[width=0.42\linewidth]{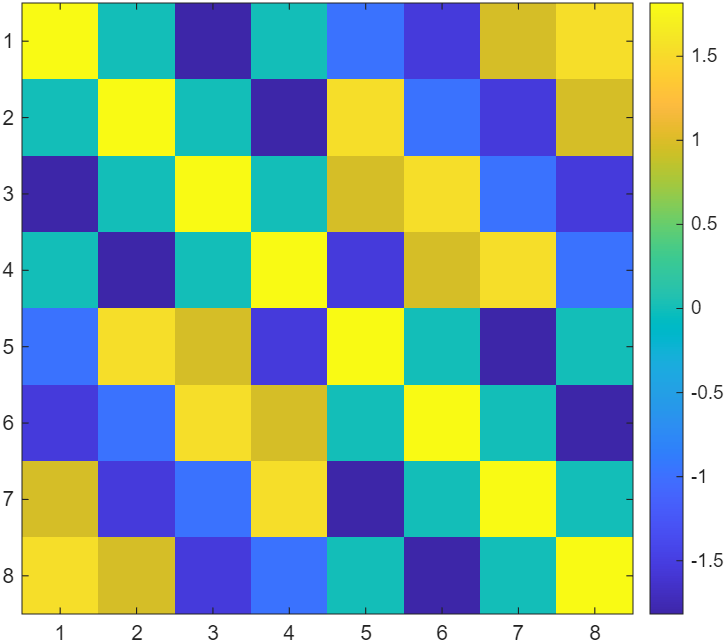}
\caption{$D_4$ experiment: the input orbit Gram matrix $G=f^{\top}f$.}
\label{fig:d4-gram}
\end{figure}

\begin{figure}[tbp]
\centering
\includegraphics[width=0.7\linewidth]{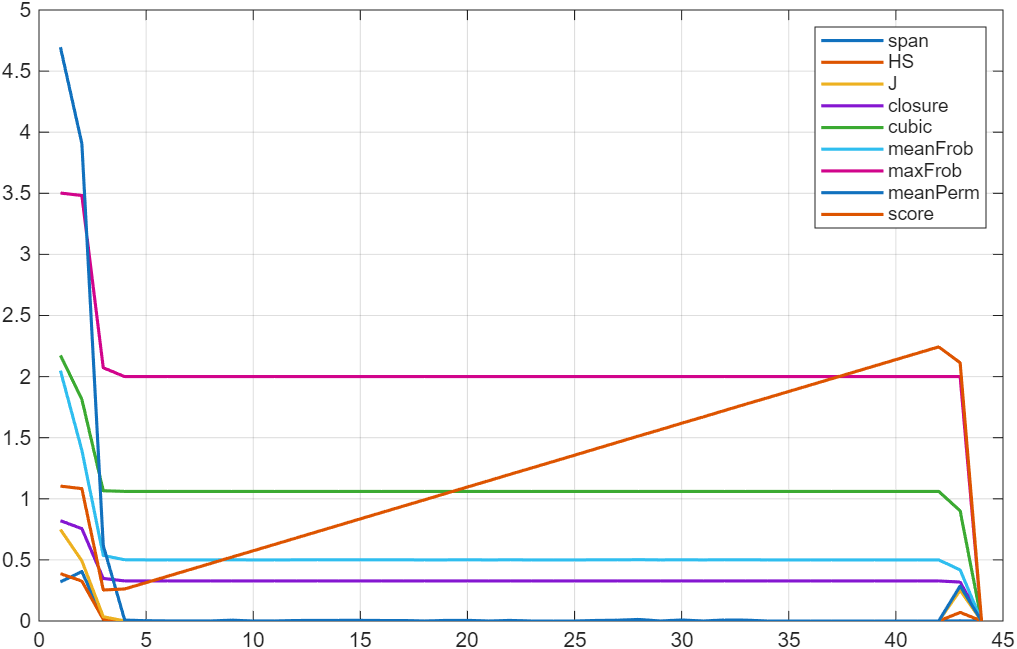}
\caption{$D_4$ recovery: metric history. All residuals collapse to machine precision.}
\label{fig:d4-history}
\end{figure}

\begin{figure}[tbp]
\centering
\begin{subfigure}[b]{0.8\linewidth}
\includegraphics[width=\linewidth]{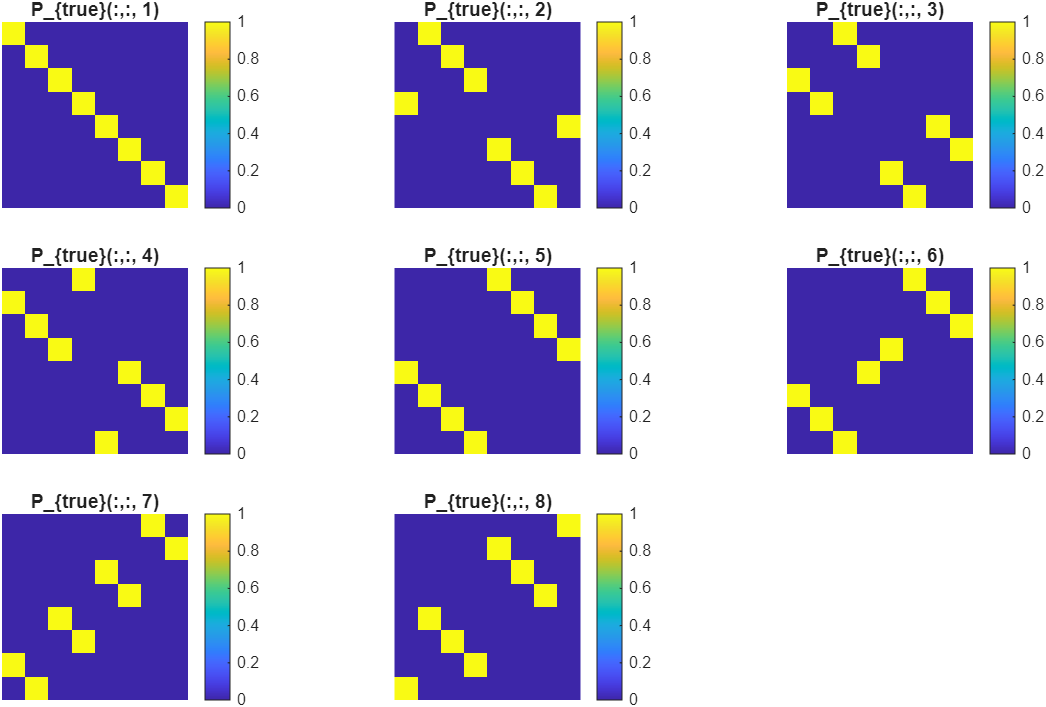}
\caption{Reference right regular representation of $D_4$.}
\label{fig:d4-true}
\end{subfigure}\\[1ex]
\begin{subfigure}[b]{0.8\linewidth}
\includegraphics[width=\linewidth]{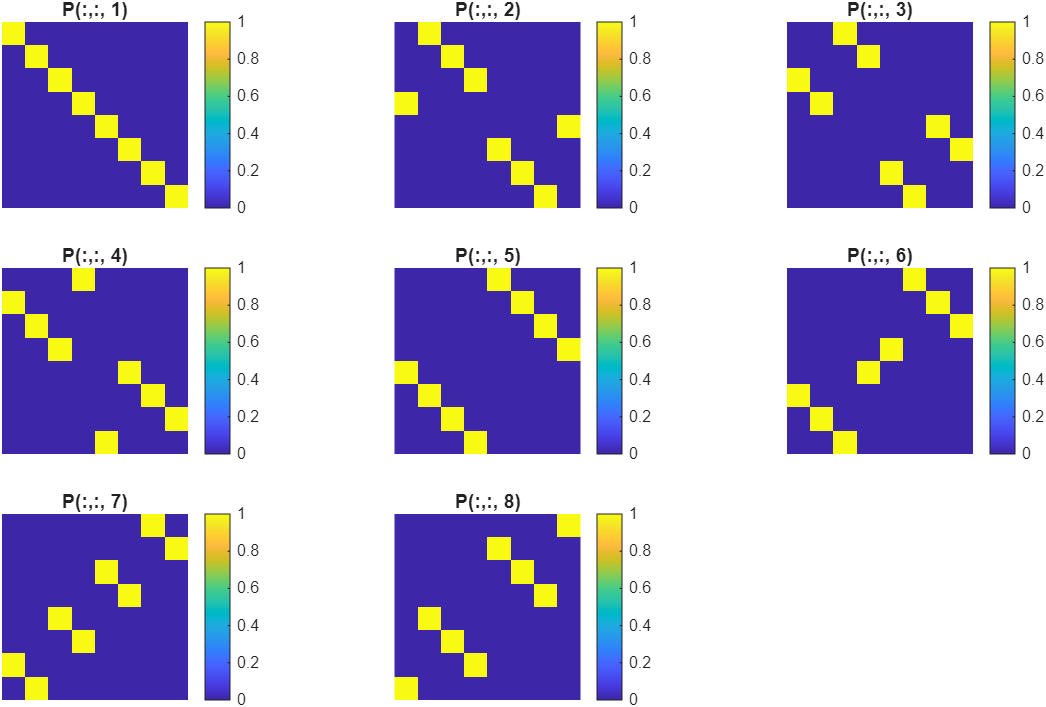}
\caption{Recovered matrices.}
\label{fig:d4-recovered}
\end{subfigure}
\caption{$D_4$: The eight reference right-regular matrices (top) and the eight recovered matrices
(bottom), shown in the output order. Each
recovered matrix is a permutation matrix coinciding with its reference counterpart to within
$10^{-15}$.}
\label{fig:d4-matrices}
\end{figure}

\FloatBarrier

\subsection{The tetrahedral rotation group \texorpdfstring{$A_4$}{A4}}\label{ssec:a4}
For $A_4$ the problem is larger ($n=12$, $d=3$) and recovery succeeds up to relabeling to
modest rather than machine precision. From the orbit Gram matrix of
Figure~\ref{fig:a4-gram}, the staged iteration reduces the structural residuals in stages
(Figure~\ref{fig:a4-history}), reaching final values
\[
\varepsilon_{\mathrm{span}}=5.1\times10^{-3},
\varepsilon_{\mathrm{HS}}=1.5\times10^{-3},
\varepsilon_{J}=6.5\times10^{-3},
\varepsilon_{\mathrm{cubic}}=2.5\times10^{-2},
\varepsilon_{\mathrm{closure}}=9.3\times10^{-3},\]
\[\varepsilon_{\mathrm{perm}}=1.1\times10^{-1}.
\]
Because the recovered family is unlabeled, a direct comparison against the reference ordering
is large (mean Frobenius error $3.27$, maximum $4.90$), purely reflecting the permuted labels:
two equal permutation matrices at different indices sit at Frobenius distance
$\sqrt{2n}\approx4.9$. Under the randomized restarts of the staged recovery strategy of
Section~\ref{sec:recovery}, which repeatedly perturbs the worst-defect matrices until every
target is met, the iteration converges to a state in which each recovered matrix corresponds
to exactly one reference matrix, already before any relabeling
(Figure~\ref{fig:a4-matching}); reindexing only sorts them into reference order, and the run
terminates on reaching the target rather than at an iteration cap. The recovery is uniform
across the family: the identity is recovered exactly, and the remaining eleven matrices lie
between $7.4\times10^{-3}$ and $4.5\times10^{-2}$ of their reference counterparts, with mean
$1.9\times10^{-2}$ (Figure~\ref{fig:a4-matrices}), so all twelve are recovered.

The recovery differs from $D_4$ in two respects. First, the residuals converge to modest
rather than machine precision: the sum of the recovered matrices matches $J$ to within
$\sim2\times10^{-2}$ (Figure~\ref{fig:a4-sumJ}), and $\varepsilon_{\mathrm{cubic}}=2.5\times10^{-2}$,
$\varepsilon_{\mathrm{perm}}=1.1\times10^{-1}$, against the $10^{-15}$ level seen for $D_4$; a
structural reason for this gap is taken up in Section~\ref{ssec:exp-disc}. Second, the recovered family arrives without labels. The index set carries no group structure
at the outset, so the order in which the matrices emerge is an artifact of the search and
means nothing on its own: there is no canonical ordering for the output to agree with. Figure~\ref{fig:a4-matching} shows the correspondence already in place before any reindexing: writing $D(g,h)=\|P_{g}-P^{\mathrm{true}}_{h}\|_F$, the dark cells forming a permutation matrix in which each recovered matrix meets exactly one
reference matrix. Reindexing applies that permutation so that the two families can be
compared, and belongs to the comparison rather than to the recovery.

\begin{figure}[tbp]
\centering
\begin{subfigure}[b]{0.42\linewidth}
\includegraphics[width=\linewidth]{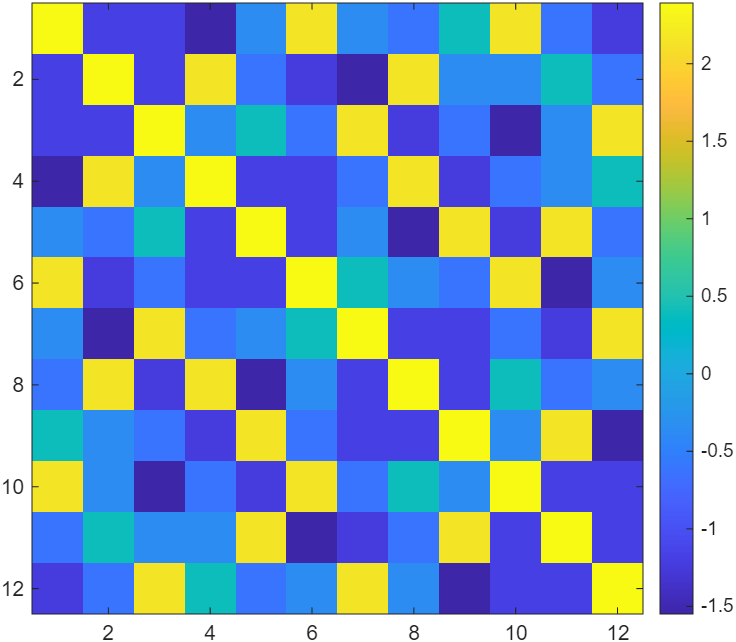}
\caption{Orbit Gram matrix $G=f^{\top}f$.}
\label{fig:a4-gram}
\end{subfigure}\hfill
\begin{subfigure}[b]{0.42\linewidth}
\includegraphics[width=\linewidth]{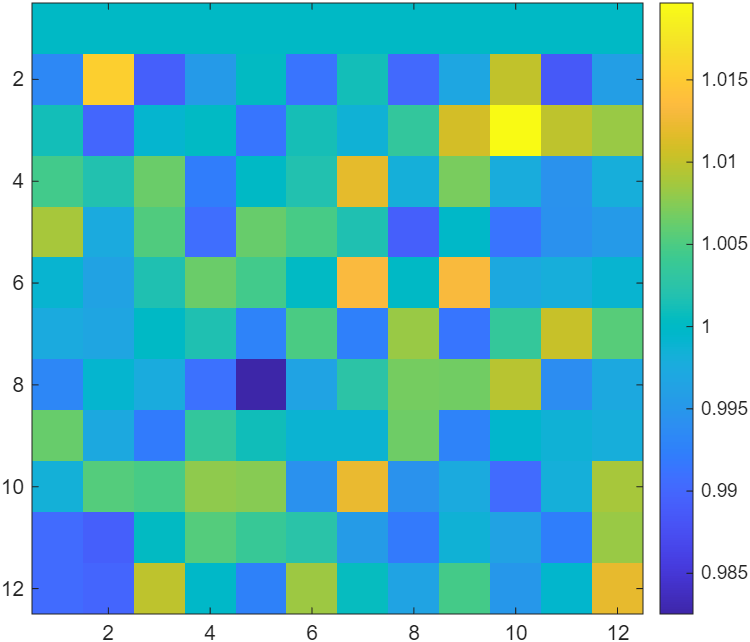}
\caption{$\sum_g P_g$, near $J$ to within $\sim2\times10^{-2}$.}
\label{fig:a4-sumJ}
\end{subfigure}
\caption{$A_4$ experiment: the input Gram matrix (left) and the sum of the recovered matrices
(right).}
\label{fig:a4-inputs}
\end{figure}

\begin{figure}[tbp]
\centering
\begin{subfigure}[b]{0.60\linewidth}
\includegraphics[width=\linewidth]{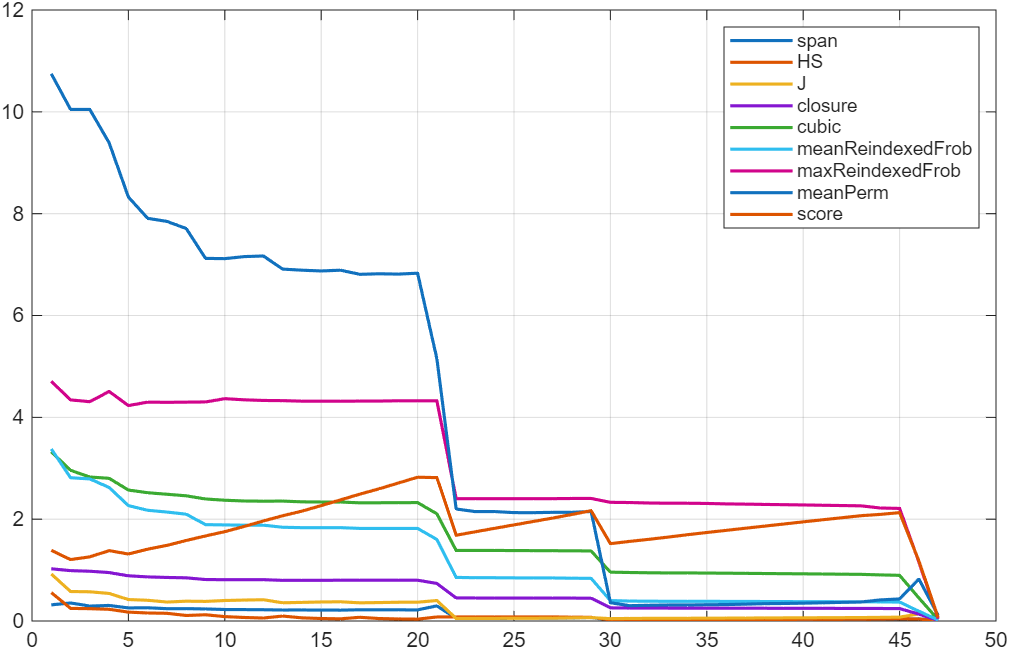}
\caption{Metric history with reindexed errors.}
\label{fig:a4-history}
\end{subfigure}\hfill
\begin{subfigure}[b]{0.36\linewidth}
\includegraphics[width=\linewidth]{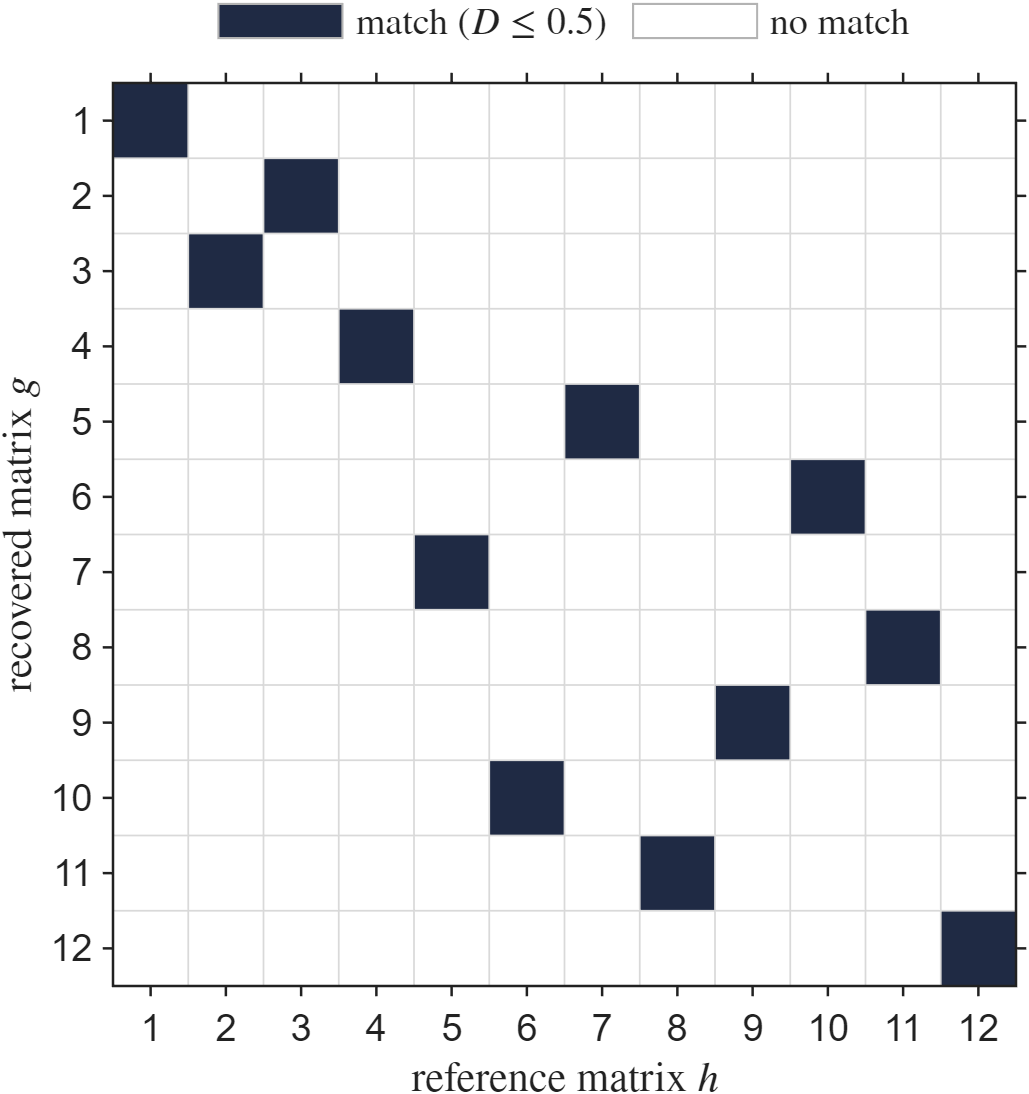}
\caption{Matching to the reference family, before reindexing.}
\label{fig:a4-matching}
\end{subfigure}
\caption{$A_4$ recovery. The residual history
(left) shows the staged decrease under the recovery strategy. The dark cells (right) form a permutation matrix, so each recovered
matrix corresponds to exactly one reference matrix before any reindexing. The identification is fixed only up to a group automorphism.}
\label{fig:a4-recovery}
\end{figure}

\begin{figure}[tbp]
\centering
\begin{subfigure}[b]{0.85\linewidth}
\includegraphics[width=0.85\linewidth]{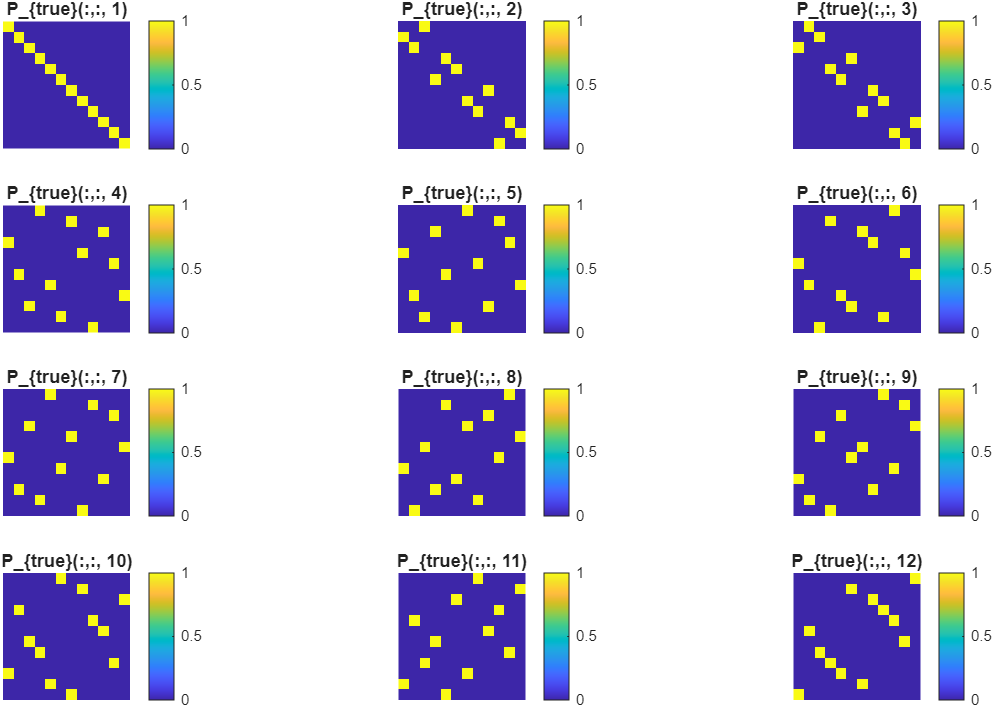}
\caption{Reference right regular representation of $A_4$.}
\label{fig:a4-true}
\end{subfigure}\\[1ex]
\begin{subfigure}[b]{0.85\linewidth}
\includegraphics[width=0.85\linewidth]{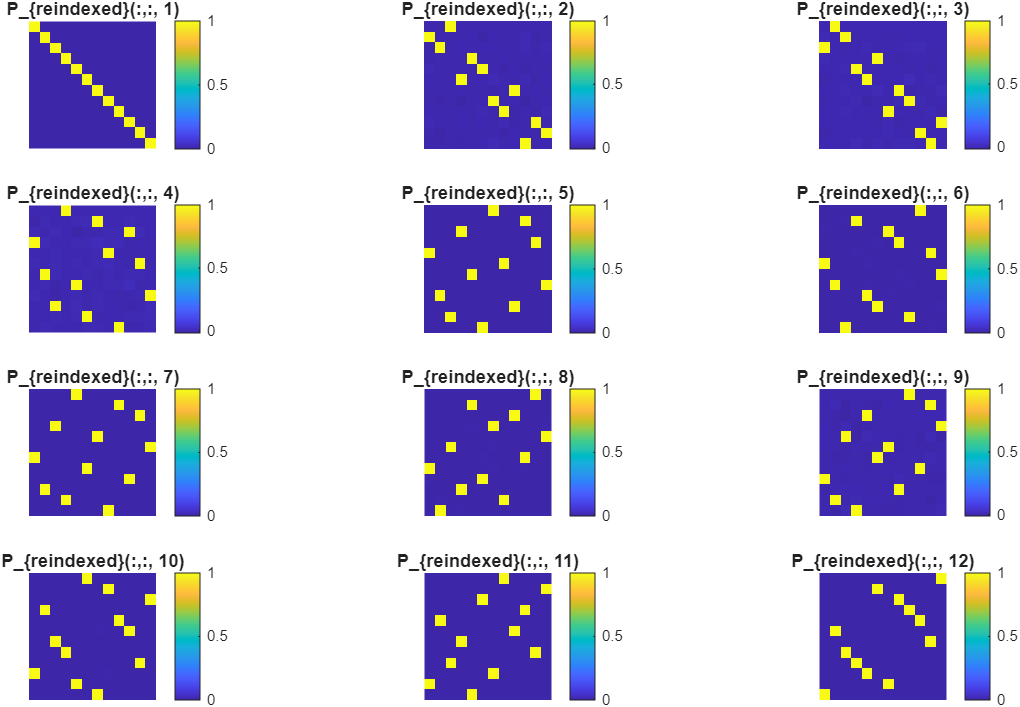}
\caption{Recovered matrices, after best-match reindexing.}
\label{fig:a4-recovered}
\end{subfigure}
\caption{$A_4$: the twelve reference right-regular matrices (top) and the reindexed recovered matrices
(bottom), in matched order. After reindexing all twelve recovered matrices are permutation
matrices coinciding with their exact counterparts, with maximum Frobenius error $4.5\times10^{-2}$.}
\label{fig:a4-matrices}
\end{figure}

\FloatBarrier

\subsection{Discussion}\label{ssec:exp-disc}
The two experiments run the same method at two scales. For $D_4$ every residual, the cubic
trace identity included, falls to near machine precision, and the recovered family coincides
with the reference family in the output order. For $A_4$ all twelve matrices are likewise recovered, up to a relabeling of the
index set, to within $4.5\times10^{-2}$ of their reference counterparts, but the
residuals settle at modest rather than machine precision. What separates the two runs is which
constraint carries the recovery. The orthogonality, span and sum-to-$J$ conditions can be met
by configurations that are not permutation representations, so satisfying them localizes the
search without settling it; the cubic trace identity
\[
\Tr(P_g P_{g'} P_h^{*})=n\,(P_{g'})_{g,h}
\]
is what distinguishes the correct configuration from these near misses. For $D_4$ the first three constraints already leave little room and the cubic identity is met; for $A_4$ the iteration has to be driven by it. Monitoring
$\varepsilon_{\mathrm{cubic}}$ is accordingly the reliable indicator of recovery quality.

A structural feature of the two orbits bears on this difference. By
Proposition~\ref{prop:schur} the right-regular algebra is closed under entrywise maps, so the
entrywise images of $G$ remain in $\mathcal{A}_R$, but they need not span it. Since
$G_{i,j}=c_{g_i^{-1}g_j}$, such a map acts only on the values of $c$, and the images span a
subspace of dimension equal to the number of distinct values. For $G$, symmetry forces
$c_g=c_{g^{-1}}$, so this count is bounded by the number of inverse pairs $\{g,g^{-1}\}$:
seven of eight for $D_4$, but only eight of twelve for $A_4$, whose three-cycles fall into
four such pairs. Among the four conditions it is the cubic trace identity that does the
decisive work for $A_4$: in our runs it is what carries the search out of a local minimum and
on to the right regular representation, to within the residuals reported above.

What separates the two runs is therefore a matter of optimization rather than of the
characterization itself: Theorem~\ref{thm:summary} applies equally to both groups, and what
differs is only how sharply the iteration can be driven toward the family it describes. These
experiments are intended to illustrate the theorem rather than to compete with methods
developed for group recovery at scale.


\section{Conclusion and future work}\label{sec:future}

We have shown that orbit-generated Gram matrices lie in the right-regular group algebra and
satisfy a cubic trace identity, and that this identity, together with Frobenius
orthogonality, a row-sum bound and a sum-to-$J$ (all-ones) constraint, is sufficient to force
a candidate family of orthogonal matrices to be the right regular representation of a finite
group. This provides a label-free criterion for abstract group recovery from a single orbit
Gram matrix, complementing the orbit-counting perspective of~\cite{mixon2024recovering}, the
group-algebra characterization of group frames in~\cite{mendez2018binary}, and the
identifiability results obtained from third-order data
in~\cite{bandeira2023estimation,edidin2025orbit}, alongside the classical fact that the group
determinant determines the group~\cite{formanek1991groups,hoehnke1992characters}. Whereas
those results show that such data \emph{distinguishes} groups or orbits, the present work
shows that a small set of trace and orthogonality conditions is \emph{sufficient} to verify
the existence of a group structure in a family that is not assumed to be a representation a
priori.

Several directions remain. The recovery strategy used here is one of many possible, and
whether an alternative reaches the right regular representation from a wider range of
starting configurations is open. It
would also be of interest to know whether the conditions are minimal, how the recovery
behaves when the Gram matrix is known only approximately, and how to treat group actions that
are linear but not isometric, in the spirit of the open problems raised
in~\cite{mixon2024recovering}.

\bibliographystyle{alpha}
\bibliography{Reference}

\end{document}